\documentclass[11pt]{amsart}
\usepackage[mathlines, displaymath]{lineno}
\usepackage[margin=1.35in]{geometry}
\usepackage{amscd,amssymb, amsmath,  setspace, wasysym, yfonts, amsthm, mathtools, dirtytalk}
\usepackage{graphicx}
\usepackage[all]{xy}
\usepackage{tikz}
\usetikzlibrary{matrix}
\usepackage{url}
\usepackage{hyperref}

\theoremstyle{plain}
\newtheorem{theorem}{Theorem}
\newtheorem{prop}[theorem]{Proposition}

\newtheorem{cor}[theorem]{Corollary}

\newtheorem{lemma}[theorem]{Lemma}

\theoremstyle{definition}

\newtheorem{defn}[theorem]{Definition}
\newtheorem{rmk}[theorem]{Remark}

\newtheorem*{ex*}{Example}

\newcommand\sO{{\mathcal O}}

\newcommand\sA{{\mathcal A}}
\newcommand\sN{{\mathcal N}}

\newcommand\sL{{\mathcal L}}

\newcommand\sU{{\mathcal U}}

\newcommand\sF{{\mathcal F}}
\newcommand\sG{{\mathcal G}}
\newcommand\sE{{\mathcal E}}
\newcommand\sI{{\mathcal I}}
\newcommand\sM{{\mathcal M}}

\newcommand\sQ{{\mathcal Q}}

\DeclareMathOperator{\Pic}{Pic}

\DeclarePairedDelimiterX\set[1]\lbrace\rbrace{#1}

\title[The decomposition theorem of Catanese, Fujita, and Kawamata]
{Singular hermitian metrics and the decomposition theorem of Catanese, Fujita, and Kawamata}

\author{Luigi Lombardi}
\address{Dipartimento di Matematica  \\ Università degli Studi di Milano Statale\\Via Cesare Saldini 50, Milan 20133, Italy}
 \email{\url{luigi.lombardi@unimi.it}}
 
\author{Christian Schnell}
\address{Department of Mathematics \\ Stony Brook University, Stony Brook, NY 11794-3651}
\email{\url{christian.schnell@stonybrook.edu}} 
 
\begin{document}
\maketitle
\begin{abstract}
We prove that a torsion-free sheaf $\sF$ endowed with a 
singular hermitian metric with semi-positive curvature and satisfying  the minimal extension property
admits a direct-sum decomposition 
   $\sF \simeq \sU \oplus \sA$ where  $\sU$ is a hermitian flat bundle  and $\sA$ is a generically 
ample sheaf. The result applies to the case of direct images of relative pluricanonical  bundles
$f_*  \omega_{X/Y}^{\otimes m}$ under a surjective morphism  $f\colon X \to Y$ of smooth projective varieties with $m\geq 2$.
This extends previous results of   Fujita,  Catanese--Kawamata, and Iwai.
 
\end{abstract}

\section{Introduction}

Let $f \colon X \to Y$ be a fibration of  smooth projective varieties (over the
complex numbers) and let $\omega_{X/Y} = \omega_X \otimes f^*\omega_Y^{-1}$ be the relative
canonical bundle.
Motivated by earlier work of
Fujita, 
 Catanese and Kawamata \cite[Theorem 1.2]{CK} proved  a direct-sum decomposition 
\begin{equation}\label{eq:ck}
f_* \omega_{X/Y} \; \simeq \; \sU \, \oplus \, \sA
\end{equation}
where $\sU$ is a hermitian flat bundle and  $\sA$
is a generically ample sheaf, if not zero.
We recall that
 a coherent
torsion-free sheaf $\sA$ on a smooth projective variety is   \emph{generically ample} if  its
restriction  to a general complete intersection smooth curve is an ample vector bundle. In dimension one we say 
that $\sA$ is generically ample if it is ample.
When $Y$ is  a smooth projective curve, the decomposition \eqref{eq:ck} is 
Fujita's   second decomposition in \cite{F}  (see also \cite[Theorem~3.3]{CD}). 
Motivated by \eqref{eq:ck}, we introduce the following definition.
\begin{defn}
A coherent torsion-free sheaf $\sF$ 
 admits a \emph{Catanese--Fujita--Kawamata decomposition} if  there exists an isomorphism
 $\sF \simeq \sU \oplus \sA$ where $\sU$ is a hermitian flat bundle, and
 $\sA$ is either a  generically ample sheaf or the zero sheaf.
\end{defn}

%

In this paper  we extend
  decomposition  \eqref{eq:ck} to direct images of relative
\emph{pluricanonical} bundles. 

\begin{theorem}\label{thm:splittingpluri}
If $f \colon X \to Y$ is a surjective morphism of smooth projective complex varieties, 
then  
$f_* \omega_{X/Y}^{\otimes m}$ admits a Catanese--Fujita--Kawamata decomposition  for every $m\geq 2$. 
\end{theorem}

When $Y$ is a smooth projective curve, Theorem \ref{thm:splittingpluri} was proved by Iwai 
in \cite[Theorem 1.4]{Iw}  by showing  in greater generality that
the reflexive hull  $ \big( f_* \omega_{X/Y}^{\otimes m} \big)^{**}$ admits a Catanese--Fujita--Kawamata decomposition
for any  smooth projective variety $Y$ and  $m\geq 1$.

While the proof of \eqref{eq:ck} is Hodge-theoretic, 
the proof of Theorem \ref{thm:splittingpluri} is based on the fact that, 
for any $m\geq 2$, the sheaf
 $f_*\omega_{X/Y}^{\otimes m}$ carries a  canonical singular hermitian metric with 
semi-positive curvature.
 Furthermore, this  metric satisfies  the so-called \emph{minimal extension property} 
(\emph{cf}. \S\ref{sec:mep}, \cite[Theorem 1.1]{PT} and  \cite[Theorem 27.1]{HPS}).
This is a property that stems from  Ohsawa--Takegoshi's extension theorem with optimal bounds
and allows one to extend local sections across subsets of
measure zero with a control on the $L^2$-norm. 
Instances of bundles satisfying the minimal extension property are
pseudo-effective line bundles and Nakano semi-positive vector bundles. 
 Theorem \ref{thm:splittingpluri} is an application of the following theorem.

\begin{theorem}\label{thm:splitting}
Let $\sF$ be a coherent torsion-free  sheaf on a smooth projective variety 
$Y$ endowed with a singular hermitian  metric
with semi-positive curvature and satisfying the minimal extension property. 
Then $\sF$ admits a Catanese--Fujita--Kawamata decomposition.
\end{theorem} 

\begin{proof}[Proof of Theorem \ref{thm:splittingpluri}]
By \cite[Theorem 1.1]{PT} and 
\cite[Theorem 27.1]{HPS} the pushforward $f_*\omega_{X/Y}^{\otimes m}$ admits a 
singular hermitian metric with semi-positive curvature and  satisfying  the minimal extension 
property for any $m\geq 2$.  The result follows by  Theorem \ref{thm:splitting}. 
\end{proof}

As an application  of Theorem \ref{thm:splittingpluri}, we prove a structure theorem 
 for the sheaves  $f_*\omega_{X/Y}^{\otimes m}$ with  $m\geq 2$.

%

\begin{theorem}\label{thm:puretype}
Let $f\colon X \to Y$ be a fibration of smooth projective  complex
varieties and denote   $J = \{ \, m \in \mathbb{N}_{\geq 2} \, \big | \, f_*\omega_{X/Y}^{\otimes m} \neq 0 \, \}$.
 Suppose   
there exists an open subset $U \subset Y$ 
such that: ${\rm codim} (Y \smallsetminus U) \geq 2$, the morphism $f$ is smooth over $U$, and
$\omega_{V/U}$ is $f|_V$-semi-ample where $V=f^{-1} (U)$.
Then   $f_*\omega_{X/Y}^{\otimes m}$ is  either generically ample for all  $m\in J$, or hermitian flat for all $m\in J$.
\end{theorem}

In \S\ref{sec:examples} 
we collect 
further instances of Catanese--Fujita--Kawamata decompositions.
\subsection*{Acknowledgments}
We thank Victor Gonz\'alez-Alonso 
for fruitful conversations. Moreover we thank  
Federico Caucci, Yajnaseni Dutta, Beppe Pareschi, Mihnea Popa and Luca Tasin
for comments on a first draft of the paper. We also thank Masataka Iwai for letting us know
about the papers \cite{HIM}, \cite{Iw} and \cite{IM}.

L.L.  was partially 
supported by GNSAGA, PSR Linea 4, and PRIN 2020: Curves, Ricci flat Varieties and their Interactions.

 During the preparation of this paper, Ch.S. was partially supported by NSF grant
 DMS-1551677 and by a grant from the Simons Foundation (817464, Schnell).
 He thanks the National Science Foundation
 and the Simons Foundation for their financial support; he also thanks the
 Max-Planck-Institute for Mathematics and the Kavli Institute for the Physics and
 Mathematics of the Universe for providing him with excellent working conditions.

\section{The minimal extension property}\label{sec:mep}

Let  $Y$ be a complex manifold of positive dimension $n$
and let 
 $( \sF , h )$ be a torsion-free sheaf  endowed with a singular hermitian  metric. 
We refer to \cite[\S 19]{HPS} for the definition  of  singular hermitian metrics with semi-positive  curvature 
on torsion-free sheaves. 
Denote   by 
 $B\subset \mathbb{C}^n$ the open unit ball centered at the origin with volume
$\mu(B) = \pi^n/n!$. Moreover  let 
$F_y$ be  the fiber of $\sF$ at a point $y\in Y$ over which  $\sF$ is locally free.
We recall the definition of the minimal extension property introduced in \cite[Definition 20.1]{HPS}.

\begin{defn}\label{def:mep}
The pair 
$(\sF,h)$ satisfies  the \emph{minimal extension property} if there exists an analytic nowhere dense  
 closed   subset $Z \subset Y$ such that $\sF$ is locally free on $Y\backslash Z$, and for every 
embedding $\iota \colon B \hookrightarrow Y$ of the open unit ball centered at  $\iota(0)=y \in Y\backslash Z$,  and vector $v\in F_y$  of 
length $\big|v\big|_{h,y} = 1$, there exists a holomorphic section $s\in H^0(B , \iota^* \sF)$  such that 
$$s(0) \,= \,  v \quad \mbox{ and }\quad \frac{1}{\mu(B)} \int_B \big| s \big|_h^2 \,\,  d\mu \, \leq \, 1.$$  
\end{defn}

We recall  a few  properties of the minimal extension property from \cite{HPS}.

\begin{prop}\label{prop:incl}
Let $(\sF , h) $ be a torsion-free sheaf  endowed with a singular hermitian  metric satisfying the minimal extension property
and let $b \colon \sF \hookrightarrow \sG$ be an inclusion of torsion-free sheaves.
If $b$ is 
generically an isomorphism, then $h$ extends to a singular hermitian metric $h_{\sG}$
 on $\sG$  satisfying  the minimal extension property.
Moreover, if $h$ has semi-positive curvature, then  
$h_{\sG}$ has semi-positive curvature as well.
 \end{prop}

\begin{proof}
The proposition is essentially proved in \cite[Proposition 19.3]{HPS} where it is stated in greater generality. 
It only remains to note that 
the minimal extension property holds for $(\sG, h_{\sG})$, but this is true   because every section of $\sF$ is also
a section of $\sG$.
\end{proof}

%

\begin{prop}\label{prop:quotient}
Let $(\sF, h )$ be a torsion-free sheaf  endowed with a singular hermitian  metric satisfying the minimal extension property. 
If $\varphi \colon \sF \twoheadrightarrow \sE$ is a quotient onto a torsion-free sheaf and $h'$ is the induced metric, then 
$(\sE,h')$ satisfies  the minimal extension property.
\end{prop}

\begin{proof}
Let $Z = Z(\sF)$ be as in Definition \ref{def:mep} and let 
$S(\sE)$ be the locus where $\sE$ is not locally free. Set $Z' = Z \cup S(\sE)$ and
let $y\in Y \backslash Z'$. 
For $w\in E_y$ the induced metric $h'$ on $\sE$ is defined by
\begin{equation}\label{eq:inducedm}
 \big| w \big|_{h',y} = {\rm inf} \{ \,  \big| v \big|_{h, y } \; \big| \; v\in F_y \mbox{ and } \varphi_y(v ) = w   \, \}.
 \end{equation}
(If $\varphi_y=0$, then the metric is $+\infty$ for all $w\neq 0$.)
Now let $w\in E_y$ be such that $ \big|w \big|_{h',y}=1$ and let  $B\subset Y$ be the embedding of the 
unit ball centered at $y\in Y \backslash Z'$.
Then there exists $v\in F_y$ such that $\big| v \big|_{h,y}=1$ 
and, by the minimal extension property of $\sF$,
a holomorphic section $s\in H^0( B, \sF|_B)$ such that $s(0)= v$ and 
$\frac{1}{ \mu(B) }\int_B \big| s \big|^2_{h} d \mu \leq 1$.
 As  $\big| \varphi(s) \big|_{h',y} \leq \big| s \big|_{h,y} $ for almost every $y\in B$, this  yields  inequalities
  $$\frac{1}{\mu(B)} \int_B  \big| \varphi(s) \big|^2_{h'} \, d\mu \, \leq \, \frac{1}{\mu(B)} \int_B \big|s \big|^2_{h} \, d \mu \, \leq \, 1.$$
\end{proof}

The following result generalizes
\cite[Theorem 26.4]{HPS}.

%

\begin{prop}\label{propsplit}
Suppose   $Y$ is a compact complex manifold 
and let $(\sF,h)$ be a torsion-free sheaf  endowed with a  singular hermitian metric  of  
semi-positive curvature and satisfying the minimal extension property. 
If $f\colon \sF \twoheadrightarrow \sU$ is a quotient onto a  vector bundle $\sU$ endowed with a  smooth  hermitian flat metric, 
then there exists a morphism $s\colon \sU \to \sF$ such that $f \circ s = {\rm id}_{\sU}$.
\end{prop}

 \begin{proof}
 Set $r = {\rm rk} \big( \sU \big)>0$.
The bundle
 $\sU$ is associated to a representation 
$\pi_1(Y) \to U( r )$ of the fundamental group of  $Y$  to the unitary group $U(r)$. 
Hence 
$\sU$ decomposes as
a direct sum of vector bundles arising  from irreducible unitary representations. 
Without loss of generality we can suppose that $\sU$ is irreducible.
  
Consider the quotients 
$$\sF \otimes \sU^{\vee} \, \twoheadrightarrow \sU\otimes \sU^{\vee} \stackrel{tr}{\twoheadrightarrow} \sO_Y$$
where
the first is induced by $f$ and 
 the second is the trace map of $\sU$. Since the metric on $\sU$, and hence
on $\sU^{\vee}$, is flat, the induced singular hermitian metric on the sheaf $\sF
\otimes \sU^{\vee}$   has semi-positive curvature and the minimal extension
property.
 By \cite[Theorem 26.4]{HPS} there exists a splitting $s'\colon \sO_Y \to \sF \otimes \sU^{\vee}$, 
and hence  a non-trivial morphism
$s''\colon \sU \to \sF$ such that $f \circ s'' \neq 0$.  By Schur's Lemma the composition 
 $f \circ s''$ is an isomorphism  and 
$s :=  s'' \circ ( f \circ s'' )^{-1} $ splits $f$.
\end{proof}

\begin{rmk}
Following \cite[Theorem 1.4]{HIM} the previous proposition remains valid 
when $\sF$ is reflexive and $h$ does not necessarily satisfy the minimal extension property.  
\end{rmk}

Finally we 	recall the following theorem  \cite[Theorem 26.1]{HPS} which is 
based on an earlier result of Cao and P\u{a}un \cite[Lemma 5.3]{CP}.
We define the determinant  of a torsion-free sheaf $\sF\neq 0$ as 
${\rm det} ( \sF) =  \big( \bigwedge^{\rm rk (\sF) } \sF \big)^{**}$.

\begin{theorem}\label{thm:c1}
Let   $Y$ be  a compact complex manifold 
and let  $(\sF,h)$ be  a nonzero  torsion-free sheaf endowed with a  singular hermitian metric. Suppose  
     $h$ has  semi-positive curvature and  satisfies  the minimal extension property. If 
$c_1 \big( \det (\sF) \big) = 0$ in $H^2(Y,\mathbb R)$, then $\sF$ is   locally free, $h$ is smooth    and $(\sF, h)$ is hermitian flat. 
\end{theorem}

\section{Catanese--Fujita--Kawamata decompositions}
%
%

\subsection{Proof of Theorem \ref{thm:splitting}}

Set $n = \dim Y$. We may suppose $n>0$ and   $\sF \neq 0$.
 By \cite[Proposition 25.1, \S26 and Definition 19.1]{HPS} the line bundle ${\rm det} ( \sF )$  admits a singular hermitian 
metric with semi-positive curvature. Hence ${\rm det} ( \sF )$  is pseudo-effective and 
for any very ample line bundle $A$ on $Y$  the \emph{degree} of $\sF$ satisfies 
$${\rm deg}_A( \sF) \; := \;  \big(  A^{n-1} \,  \cdot \,  \sF \big)    \; \geq \; 0$$
by Nakai--Moishezon's Theorem.
We define the $A$\emph{-slope} of   $\sF$   as
$$\mu^A(\sF) \, := \,  \frac{\deg_A(\sF)} { {\rm rk}(\sF)} $$ and 
 say that it  is \emph{semistable} (with respect to $A$) if for  every nonzero coherent sub-module $\sE \subset \sF$  the inequality  
$\mu^A(\sE) \leq \mu^A(\sF)$ holds.
Let
\begin{equation}\label{eq:HN}
0 \, = \,  \sN_0 \, \subsetneq \, \sN_1 \, \subsetneq  \, \ldots \, \subsetneq\,  \sN_d \, = \, \sF
\end{equation}
be the Harder--Narasimhan filtration of $\sF$. Hence
for any $i=1,\ldots ,d$
 the quotients $\sN_i / \sN_{i-1}$ are torsion-free  and semistable.
Moreover  the slopes $$\mu^A_i := \mu^A (\sN_i / \sN_{i-1})$$ satisfy
$$\mu^A_1 \, > \, \mu^A_2 \, > \, \ldots \,  >\,  \mu^A_d $$
\cite[Proposition-Definition 1.13]{Ma}.
We denote by
$$\sQ\; := \; \sF / \sN_{d-1} \; = \;  \sN_d / \sN_{d-1}$$
the minimal destabilizing quotient of $\sF$ and set
 $$\sL := {\rm det}(\sQ)$$ for the determinant of $\sQ$.
By \cite[Lemma 2.4.3]{PT} and \cite[Proposition 25.1]{HPS} both
$\sQ$ and  $\sL$ admit
a singular hermitian metric with semi-positive curvature. 
Hence 
$\sL$ is  pseudo-effective   and  ${\rm deg}_A(\sL) \geq 0$.
We distinguish    two cases:  ${\rm deg}_A(\sL)>0$ and ${\rm deg}_A(\sL) =0$.

Let's begin with the   case ${\rm deg}_A(\sL) >0$. 
We are going to show that $\sF$ is already generically ample.
Let $H = A^{\otimes a}$ be a very ample line bundle
with $a\gg 0$   so that Flenner's Theorem \cite[Theorem 7.1.1]{HL} applies
to a general complete intersection smooth curve $C$ cut out by divisors in $| H |$ and contained 
in the locus where $\sF$ is locally free.
It follows that the Harder--Narasimhan filtration \eqref{eq:HN} of $\sF$ restricts to the Harder--Narasimhan filtration
$$ 
0 \, = \,  \sM_0 \, \subsetneq \, \sM_1 \, \subsetneq  \, \ldots \, \subsetneq\,  \sM_d \, = \, \sF|_C$$
of $\sF|_C$. Here the sheaves  $\sM_i := \sN_i|_C$ are locally free and  semistable, 
and 
$$ \frac{{\rm deg} \big( \sM_i / \sM_{i-1} \big) }{{\rm rk }\big( \sM_i / \sM_{i-1} \big)  } > 
 \frac{{\rm deg} \big( \sM_{i+1} / \sM_{i} \big) }{{\rm rk }\big( \sM_{i+1} / \sM_{i} \big)  }$$
 for all $i=1, \ldots , d-1$.
Since 
$${\rm deg}\big(\sM_i / \sM_{i-1} \big) \, = \, \Big(   H^{n-1} \,  \cdot \,  \sN_i / \sN_{i-1}   \Big) \, = \,
a^{n-1}\Big( A^{n-1} \, \cdot \,   \sN_i / \sN_{i-1}  \Big),
$$ 
the \emph{minimal slope}  of $\sF|_C$ satisfies 
$$\mu_{{\rm min}} \big( \sF|_C \big) \, := \,  \frac{{\rm deg} \big( \sM_{d} / \sM_{d-1} \big) }{{\rm rk }\big( \sM_{d} / \sM_{d-1} \big)  } \, = \, a^{n-1} \mu_d^A \, = \, a^{n-1} \frac{{\deg}_A ( \sL)}{{\rm rk} (\sQ)} \, > \, 0.$$
 By \cite[Theorem 2.1]{Br} $\sF|_C$ is an  ample bundle. 
In this case we set $\sU=0$ and $\sA = \sF$.

Now let us  suppose that   ${\rm deg}_A(\sL) =0$. 
We will first  show that $c_1(\sL) =0$ in  $H^2(Y , \mathbb R)$.
Let $H=A^{\otimes a}$ be a very ample line bundle as before with $a$ sufficiently large.
Let $D_1 , \ldots ,D_{n-2} \in | H |$ be general members such that 
for all $i=2, \ldots , n-2$
each partial  intersection
$$V_i\, := \,  D_1 \, \cap \, D_2 \, \cap \, \ldots \, \cap \, D_i$$ is a smooth and irreducible ample divisor in $V_{i-1}$. 
Hence  $S:= V_{n-2}$ is a smooth  surface with $[S] = H^{n-2}$ in 
$H_4 ( Y ; \mathbb Z)$ such that  
the restriction map
\begin{equation}\label{eq:lef}
H^2( Y ; \mathbb Z) \to H^2 (S ; \mathbb Z) \quad \mbox{is injective}
\end{equation}
(see  Lefschetz's Hyperplane Theorem \cite[Theorem 3.1.17]{Laz}).
Moreover, we can choose the general divisors $D_i \in | H |$ such that the restrictions $\sL|_{V_i}$ are pseudo-effective.
Hence
$$0 \; = \;  \big(    H^{n-1} \, \cdot \, \sL \big) \;  = \;  \big( H|_S \, \cdot \, \sL|_S \big)  $$ 
from which we deduce that $\sL|_S$ is numerically trivial
as $H|_S$ is ample and $\sL|_S$ is a limit of effective classes (\emph{cf}.
\cite[Theorem 1.4.29]{Laz}).  By \eqref{eq:lef}  the claim follows.

Since $c_1 ( \sL) = 0$, by
 Proposition \ref{prop:quotient} 
and Theorem \ref{thm:c1} the bundle  $\sQ$ is  hermitian flat.
Moreover, by Proposition \ref{propsplit}  
there exists a  decomposition  $\sF \simeq \sQ \oplus \sN_{d-1}$
so that we only need to  prove that $\sN_{d-1}$ is generically ample (or zero).
Let  $C$ be a general complete intersection smooth curve cut out by divisors in $| H |$ and consider the 
Harder--Narasimhan filtration 
$$0 \, = \, \sM_0 \, \subsetneq \, \sM_1 \, \subsetneq \, \ldots \, \subsetneq \, \sM_{d-2} \, \subsetneq \, \sM_{d-1}, $$
of $\sM_{d-1} := \big( \sN_{d-1}\big)|_C$
where, as before, $\sM_i := \sN_i|_C$.
As the minimal slope of $\sM_{d-1}$ satisfies
$$\mu_{\rm min} \big( \sM_{d-1} \big) \, := \, 
\frac{\deg \big( \sM_{d-1} / \sM_{d-2} \big)}{{\rm rk} \big( \sM_{d-1} / \sM_{d-2} \big) } \, = \, a^{n-1} \mu_{d-1}^A \, > \, 
a^{n-1} \mu_d^A \, = \, 
a^{n-1} \frac{{\rm deg}_A ( \sL) }{{\rm rk} ( \sQ) }\, = \, 0,$$
the bundle $\sM_{d-1}$ is ample. In this case we set $\sU=\sQ$ and $\sA = \sN_{d-1}$.


%

 \section{Further applications}\label{sec:examples}

\begin{cor}
Let $f\colon X \to Y$ be a surjective morphism of smooth projective varieties and let $L$ be a  semiample line bundle on $X$.
Then $f_*(\omega_{X/Y} \otimes L)$ admits a Catanese--Fujita--Kawamata decomposition.
\end{cor} 
 
 \begin{proof}
%
%
%

 As a consequence of the semiampleness, the line bundle $L$
  admits a smooth hermitian metric $h$ with
 semi-positive curvature; in particular, the multiplier ideal $\sI(h) = \sO_X$ is
 trivial.  By P\u{a}un--Takayama's theorem  \cite[Theorem 1.1]{PT} and \cite[Theorem 21.1]{HPS},  
 $f_* ( \omega_{X/Y} \otimes L)$ admits a singular hermitian metric
 with semi-positive curvature and  the minimal extension property.
The corollary   follows from Theorem \ref{thm:splitting}.
 \end{proof} 
 
%
%
%

\begin{cor}\label{thm:splittinghigher}
If $f \colon X \to Y$ is a smooth fibration of smooth projective   varieties and $E$ is a Nakano semi-positive vector bundle on $X$,
then  
$ R^j f_* ( \omega_{X/Y} \otimes E)$ 
admits a Catanese--Fujita--Kawamata decomposition for every  $j\geq 0$. 
\end{cor}
\begin{proof}
The main result of \cite[Theorem 1.1]{MT} proves that $R^j f_* ( \omega_{X/Y} \otimes E)$ is a Nakano semi-positive vector bundle 
for all $j\geq 0$. 
 On the other hand a Nakano semi-positive vector bundle satisfies the minimal extension property (\emph{cf}. \cite[Example 2.16]{SY}).
The  result follows by Theorem \ref{thm:splitting}.
 
\end{proof}

 \begin{rmk}
 As a further example, 
 Iwai and Matsumura prove that a nef  cotangent bundle $\Omega_X^1$ of 
a smooth projective variety
admits a Catanese--Fujita--Kawamata decomposition (\emph{cf}. \cite[Proposition 4.2]{IM}).
  \end{rmk}
 
 \begin{rmk}
 Let $f\colon X \to A$ be a surjective morphism from a smooth projective complex variety $X$ 
 to an abelian variety $A$. By employing \cite[Theorem C]{LPS} it is possible to  check that 
  the hermitian flat part
 of the Catanese--Fujita--Kawamata decomposition of 
  $f_*\omega_{X/A}^{\otimes m}$ ($m\geq 1$) consists  of a  direct-sum of finitely many 
 torsion line bundles in $\Pic^0(A)$.

 \end{rmk}

%

\section{Proof of Theorem \ref{thm:puretype}} 
 
In this section we  prove Theorem \ref{thm:puretype}.
In order to do so,   
we begin by recalling the following result of Esnault and Viehweg, which in fact holds 
under more general assumptions.

Given a fibration $f\colon X \to Y$ of varieties, a line bundle $L$ on $X$ is $f$\emph{-semi-ample}
if for some positive integer $N>0$ the natural morphism $f^*f_*L^{\otimes N} \to L^{\otimes N}$ is surjective.

\begin{theorem}\label{prop:det}
Let $g \colon X \to C$ be a smooth fibration from a smooth projective variety to a smooth projective curve such that $\omega_{X/C}$ is $g$-semi-ample.
If ${\rm det } ( g_* \omega_{X/C}^{\otimes m} )$ is an ample line bundle for some $m\geq 1$, then for all $m\geq 2$ the bundle 
$ g_* \omega_{X/C}^{\otimes m} $ is  ample, if not the zero sheaf.
\end{theorem}
\begin{proof}
The proof is a special case of 
\cite[Theorem 0.1]{EV}.
\end{proof}

\begin{cor}\label{cor:purec}
Let $g\colon X \to C$ be a smooth fibration from a smooth projective variety to a smooth projective curve
such that $\omega_{X/C}$ is $g$-semi-ample.
Denote  $J = \{ m \in \mathbb{N}_{\geq 2} \, \big | \, g_*\omega_{X/C}^{\otimes m} \neq 0 \}$.
Then either $g_*\omega_{X/C}^{\otimes m}$ is  ample for every $m\in J$, or  hermitian flat for every  $m\in J$.
 \end{cor}

\begin{proof}
By Theorem \ref{thm:splitting} 
 $g_*\omega_{X/C}^{\otimes m}$  decomposes
as $g_*\omega_{X/C}^{\otimes m} \simeq \sU_m \oplus  \sA_m$ with $\sU_m$  hermitian flat (or zero),  and 
$\sA_m$  ample (or zero). 
If $\sA_m =0$ 
for all $m\in J$, then $g_*\omega_{X/C}^{\otimes m}$ is hermitian flat for every  $m\in J$ and the proof is complete. 
On the
other hand, if 
$\sA_m \neq 0$ for some $m\in J$,
then ${\rm det} \big( g_*\omega_{X/C}^{\otimes m} \big) $ is ample, and 
by Theorem \ref{prop:det} the bundle $g_*\omega_{X/C}^{\otimes m}$    is itself  ample  for every $m\in J$.
\end{proof}

%
%

\begin{lemma}\label{lem:hyperplane}
Let $f\colon X \to Y$ be a fibration of smooth projective varieties.
Then 
for  any sufficiently general hyperplane section $H \subset Y$ the variety  $X_H := f^{-1}(H)$
is smooth and irreducible. Moreover, if $m$ is a positive integer, then there is an isomorphism 
$f_* \omega_{X/Y}^{\otimes m} |_H \simeq g_* \omega_{X_H / H }^{\otimes m}$
where $g := f|_{X_H}$.
\end{lemma}

\begin{proof}
The fact that $X_H$ is smooth and irreducible follows by  Bertini theorem, as 
stated in  \cite[Theorem 6.3]{J}. 
The following argument is inspired by \cite{Ko}. 
Consider the following cartesian diagram
$$
\centerline{ \xymatrix@=32pt{
 X_H \ar@{^{(}->}[r]^{j} \ar[d]^{g} & X \ar[d]^f \\
 H \ar@{^{(}->}[r]^{i} &  Y \\}} 
   $$
where $i$ and $j$ are the natural inclusions.
There is a short exact sequence for any $m\geq 1$
$$0 \to \omega_{X/Y}^{\otimes m} \otimes \sO_X (-X_H) \to \omega_{X/Y}^{\otimes m} \to j_*\omega_{X_H /H}^{\otimes m} \to 0$$
since  $N_{X_H / X} \simeq g^* N_{H/Y}$ and $\omega_{X/Y}^{\otimes m}|_{X_H} \simeq \omega_{X_H/H}^{\otimes m}$. 
By taking higher direct images, there is    a long exact sequence 
$$0 \to f_*   \omega_{X/Y}^{\otimes m} \otimes \sO_Y (-H)   \to f_* \omega_{X/Y}^{\otimes m} \to i_*g_*\omega_{X_H /H}^{\otimes m} \to$$
$$\to R^1f_*  \omega_{X/Y}^{\otimes m} \otimes \sO_Y (-H)   \to R^1 f_* \omega_{X/Y}^{\otimes m} \to i_* R^1 g_*\omega_{X_H /H}^{\otimes m} \to \ldots
$$
$$ \ldots \to 
R^jf_*   \omega_{X/Y}^{\otimes m} \otimes \sO_Y (-H)   \to R^j f_* \omega_{X/Y}^{\otimes m} \to i_* R^j g_*\omega_{X_H /H}^{\otimes m} 
\to \ldots
$$
For every  index $j\geq 0$ the morphism
$\psi_j \colon R^jf_*  \omega_{X/Y}^{\otimes m} \otimes \sO_Y (-H)   \to R^j f_* \omega_{X/Y}^{\otimes m}$ is
 obtained by tensoring 
the natural inclusion $\sO_Y (-H)   \to \sO_Y$ with  $R^jf_*  \omega_{X/Y}^{\otimes m}$. Hence, for a general 
hyperplane section $H$,  which does not contain any associated subvariety of $R^jf_*  \omega_{X/Y}^{\otimes m}$, the morphism
$\psi_j$ is injective. For $j=0$ it follows that the following sequence
\begin{equation}\label{eq:seqH}
0 \to f_*   \omega_{X/Y}^{\otimes m} \otimes \sO_Y (-H)   \to f_* \omega_{X/Y}^{\otimes m} \to i_*g_*\omega_{X_H /H}^{\otimes m} \to 0
\end{equation}
is short exact.  The desired isomorphism is obtained by restricting \eqref{eq:seqH} to $H$.
\end{proof}

 Repeated applications of the previous lemma yield the following corollary.

\begin{cor}\label{cor:genbc}
Let $f\colon X \to Y$ be a fibration of smooth projective varieties.
 Then for a general complete intersection smooth curve $C$ in $Y$ the 
variety $X_C = f^{-1}(C)$ is  smooth  and 
irreducible. Moreover, 
if $m$ is a positive integer, then 
there is an isomorphism
$f_* \omega_{X/Y}^{\otimes m} |_C \simeq g_* \omega_{X_C /C }^{\otimes m}$
where $g := f|_{X_C}$.
\end{cor}

\begin{proof}[Proof of Theorem \ref{thm:puretype}]
Set $n = \dim Y$.
Thanks to Corollary \ref{cor:purec} 
we can assume that $n \geq 2$. Moreover,  
without loss of generality we can suppose that $J\neq \emptyset$. 
Recall that for every $m\in J$ there exists
a decomposition 
\begin{equation}\label{eq:recall}
\sF_m \, := \,  f_* \omega_{X/Y}^{\otimes m} \, \simeq \,  \sU_m \oplus \sA_m
\end{equation}
 where 
$\sU_m$ is hermitian flat (or zero), and $\sA_m$ is generically ample (or zero).
If $\sA_m = 0$ for all $m\in J$, then  
$\sF_m$ is hermitian flat for every  $m\in J$ and the proof is complete.
We may suppose that there exists an index $m_0\in J$ such that $\sA_{m_0} \neq 0$.
We aim to prove that  $\sU_m= 0$ for all $m\in J$.

Let $U\subset Y$ be an open subset 
as in the statement of the theorem 
and fix a very ample line bundle $H$ on $Y$. Moreover,  let 
$C\subset U$ be a general complete intersection smooth curve cut out by divisors in $|H|$ 
such that  ${\rm deg} ( \sA_{m_0}|_C ) >0$.
If $g\colon X_C \to C$ denotes the restriction of $f$ to $X_C := f^{-1} ( C )$, then 
by Corollary \ref{cor:genbc} there is an isomorphism
\begin{equation}\label{eq:bcres}
0 \; \neq \; \sF_{m_0}|_C  \;  \simeq \; g_*\omega_{X_C/C}^{\otimes m_0} \; \simeq \; \sA_{m_0}|_C \oplus \sU_{m_0}|_C.
\end{equation}
Hence we have
$${\rm deg}  \big( \det ( g_*\omega_{X_C/C}^{\otimes m_0} )  \big) \; = 
\;   \big( {\rm det} ( \sA_{m_0})  \, \cdot \, C \big) \;> \; 0$$
and 
by Corollary \ref{cor:purec} we conclude  that
\begin{equation}\label{eq:genample}
g_*\omega_{X_C/C}^{\otimes k} 
\mbox{ is ample for all } k\in J.
\end{equation}
If $m\in J$ is an arbitrary index, we can select a curve $C$ (depending on $m$) as above, and general enough so
that also the following isomorphism
\begin{equation}\label{eq:finalrestr}
0 \, \neq\, \sF_m|_C \;  \simeq \; g_*\omega_{X_C/C}^{\otimes m}
\end{equation}
holds.
By \eqref{eq:recall}, \eqref{eq:genample} and 
\eqref{eq:finalrestr} the bundle $g_*\omega_{X_C/C}^{\otimes m}\simeq \sU_{m}|_C \oplus \sA_{m}|_C$
is ample, 
and therefore $\sU_m|_C$ itself is ample, if not trivial.
Since $\big(  {\rm det} (\sU_m )\,  \cdot \,  C \big) =0$, this forces ${\rm det} (\sU_m)|_C=0$, and thus $\sU_m =0$.
 \end{proof}

\bibliographystyle{amsalpha}
\bibliography{bibl}

\providecommand{\bysame}{\leavevmode\hbox to3em{\hrulefill}\thinspace}
\providecommand{\MR}{\relax\ifhmode\unskip\space\fi MR }
\providecommand{\MRhref}[2]{%
  \href{http://www.ams.org/mathscinet-getitem?mr=#1}{#2}
}
\providecommand{\href}[2]{#2}
\begin{thebibliography}{HIM22}

\bibitem[Bre04]{Br}
Holger Brenner, \emph{Slopes of vector bundles on projective curves and
  applications to tight closure problems}, Trans. Amer. Math. Soc. \textbf{356}
  (2004), no.~1, 371--392.

\bibitem[CD17]{CD}
Fabrizio Catanese and Michael Dettweiler, \emph{Answer to a question by
  {F}ujita on variation of {H}odge structures}, Higher dimensional algebraic
  geometry---in honour of {P}rofessor {Y}ujiro {K}awamata's sixtieth birthday,
  Adv. Stud. Pure Math., vol.~74, Math. Soc. Japan, Tokyo, 2017, pp.~73--102.

\bibitem[CK19]{CK}
Fabrizio Catanese and Yujiro Kawamata, \emph{Fujita decomposition over higher
  dimensional base}, Eur. J. Math. \textbf{5} (2019), no.~3, 720--728.

\bibitem[CP17]{CP}
Junyan Cao and Mihai P\u{a}un, \emph{Kodaira dimension of algebraic fiber
  spaces over abelian varieties}, Invent. Math. \textbf{207} (2017), no.~1,
  345--387.

\bibitem[EV91]{EV}
H\'{e}l\`ene Esnault and Eckart Viehweg, \emph{Ample sheaves on moduli
  schemes}, Algebraic geometry and analytic geometry ({T}okyo, 1990), ICM-90
  Satell. Conf. Proc., Springer, Tokyo, 1991, pp.~53--80.

\bibitem[Fuj78]{F}
Takao Fujita, \emph{The sheaf of relative canonical forms of a {K}\"{a}hler
  fiber space over a curve}, Proc. Japan Acad. Ser. A Math. Sci. \textbf{54}
  (1978), no.~7, 183--184.

\bibitem[HIM22]{HIM}
Genki Hosono, Masataka Iwai, and Shin-ichi Matsumura, \emph{On projective
  manifolds with pseudo-effective tangent bundle}, J. Inst. Math. Jussieu
  \textbf{21} (2022), no.~5, 1801--1830.

\bibitem[HL10]{HL}
Daniel Huybrechts and Manfred Lehn, \emph{The geometry of moduli spaces of
  sheaves}, second ed., Cambridge Mathematical Library, Cambridge University
  Press, Cambridge, 2010.

\bibitem[HPS18]{HPS}
Christopher Hacon, Mihnea Popa, and Christian Schnell, \emph{Algebraic fiber
  spaces over abelian varieties: around a recent theorem by {C}ao and
  {P}\u{a}un}, Local and global methods in algebraic geometry, Contemp. Math.,
  vol. 712, Amer. Math. Soc., [Providence], RI, [2018] \copyright 2018,
  pp.~143--195.

\bibitem[IM22]{IM}
Masataka Iwai and Shin-ichi Matsumura, \emph{Abundance theorem for minimal
  compact {K}\"ahler manifolds with vanishing second {C}hern class},
  arXiv:2205.10613 (2022).

\bibitem[Iwa22]{Iw}
Masataka Iwai, \emph{Almost nef regular foliations and {F}ujita's decomposition
  of reflexive sheaves}, Ann. Sc. Norm. Super. Pisa Cl. Sci. (5) \textbf{23}
  (2022), no.~2, 719--743.

\bibitem[Jou83]{J}
Jean-Pierre Jouanolou, \emph{Th\'{e}or\`emes de {B}ertini et applications},
  Progress in Mathematics, vol.~42, Birkh\"{a}user Boston, Inc., Boston, MA,
  1983.

\bibitem[Kol86]{Ko}
J\'{a}nos Koll\'{a}r, \emph{Higher direct images of dualizing sheaves. {I}},
  Ann. of Math. (2) \textbf{123} (1986), no.~1, 11--42.

\bibitem[Laz04]{Laz}
Robert Lazarsfeld, \emph{Positivity in algebraic geometry. {I}}, Ergebnisse der
  Mathematik und ihrer Grenzgebiete. 3. Folge. A Series of Modern Surveys in
  Mathematics [Results in Mathematics and Related Areas. 3rd Series. A Series
  of Modern Surveys in Mathematics], vol.~48, Springer-Verlag, Berlin, 2004,
  Classical setting: line bundles and linear series.

\bibitem[LPS20]{LPS}
Luigi Lombardi, Mihnea Popa, and Christian Schnell, \emph{Pushforwards of
  pluricanonical bundles under morphisms to abelian varieties}, J. Eur. Math.
  Soc. (JEMS) \textbf{22} (2020), no.~8, 2511--2536.

\bibitem[Mar81]{Ma}
Masaki Maruyama, \emph{The theorem of {G}rauert-{M}\"{u}lich-{S}pindler}, Math.
  Ann. \textbf{255} (1981), no.~3, 317--333.

\bibitem[MT08]{MT}
Christophe Mourougane and Shigeharu Takayama, \emph{Hodge metrics and the
  curvature of higher direct images}, Ann. Sci. \'{E}c. Norm. Sup\'{e}r. (4)
  \textbf{41} (2008), no.~6, 905--924.

\bibitem[PT18]{PT}
Mihai P\u{a}un and Shigeharu Takayama, \emph{Positivity of twisted relative
  pluricanonical bundles and their direct images}, J. Algebraic Geom.
  \textbf{27} (2018), no.~2, 211--272.

\bibitem[SY23]{SY}
Christian Schnell and Ruijie Yang, \emph{Hodge modules and singular hermitian
  metrics}, Math. Z. \textbf{303} (2023), no.~2, Paper No. 28, 20.

\end{thebibliography}

\end{document}